\documentclass[reqno,centertags,a4paper,12pt]{amsart}
\usepackage{amssymb}
\usepackage{amsthm}
\usepackage{eucal}
\usepackage{color, soul, xcolor}
\soulregister\cite{7}
\soulregister\ref{7}
\soulregister\pageref7
\usepackage[english]{babel}
\usepackage{graphicx}
\usepackage{subfig}

\newcommand{\R}{\mathbb R}

\newcommand{\Z}{\mathbb Z}
\newcommand{\T}{\mathbb T}

\newcommand{\sgn}{\text{sgn}}

\numberwithin{equation}{section}
\newtheorem{theorem}{Theorem}[section]

\newtheorem{remark}[theorem]{Remark}
\newtheorem{lemma}[theorem]{Lemma}
\newtheorem{corollary}[theorem]{Corollary}

\newtheorem*{TA}{Theorem A}

\begin{document}
%\today
\title[Propagation of Regularities and Radius of Analyticity]{On propagation of regularities and evolution of  radius of analyticity  in the solution of the  fifth order KdV-BBM model}

\author{X. Carvajal}
\address{Instituto de Matem\'atica, UFRJ, 21941-909, Rio de Janeiro, RJ, Brazil}
\email{carvajal@im.ufrj.br}
\author{M. Panthee}
\address{IMECC-UNICAMP\\
13083-859, Campinas, S\~ao Paulo, SP,  Brazil}
\email{mpanthee@unicamp.br}

%\thanks{This work was partially supported by FAPESP Brazil and the University of Illinois at Chicago}

\begin{abstract}
 We consider the initial value problem (IVP) associated to a fifth order KdV-BBM type model that describes the propagation of unidirectional water waves. We prove that  the regularity in the initial data propagates in the solution, in other words no singularities can appear or disappear in the solution to this model. We also prove the local well-posedness of the IVP in the space of the analytic functions, the so called Gevrey class. Furthermore, we discuss the evolution of radius of analyticity in such class by providing explicit formulas for upper and lower bounds. 
\end{abstract}

\maketitle

% Heading 1
\section{Introduction}

In this work we consider the following  initial value problem (IVP) 
\begin{equation}\label{5kdvbbm}
\begin{cases}
\eta_t+\eta_x-\gamma_1 \eta_{xxt}+\gamma_2\eta_{xxx}+\delta_1 \eta_{xxxxt}+\delta_2\eta_{xxxxx}+\frac{3}{2}\eta \eta_x+\gamma (\eta^2)_{xxx}-\frac{7}{48}(\eta_x^2)_x-\frac{1}{8}(\eta^3)_x=0,   \\
\eta(x,0)=\eta_0(x),
\end{cases} 
\end{equation}
where
\begin{equation}\label{parametros}
\gamma_1=\frac{1}{2}(b+d-\rho),\qquad 
\gamma_2=\frac{1}{2}(a+c+\rho), 
\end{equation}
with $\rho = b+d-\frac16$, and 
\begin{equation}\label{parametros-1}
\begin{cases}
\delta_1=\frac{1}{4}\,[2(b_1+d_1)-(b-d+\rho)(\frac{1}{6}-a-d)-d(c-a+\rho)], \\
\delta_2 =\frac{1}{4}\,[2(a_1+c_1)-(c-a+\rho)(\frac{1}{6}-a)+\frac{1}{3}\rho], 
\gamma =\frac{1}{24}[5-9(b+d)+9\rho].
\end{cases}
\end{equation}
The parameters appearing in \eqref{parametros} and \eqref{parametros-1}  satisfy $a+b+c+d=\frac{1}{3}$, $\gamma_1+\gamma_2=\frac{1}{6}$, $\gamma=\frac{1}{24}(5-18\gamma_1)$ and $\delta_2-\delta_1=\frac{19}{360}-\frac{1}{6}\gamma_1$ with $\delta_1>0$ and $\gamma_1>0$.

We have two main objectives in this work. First, we plan to study the propagation of regularities in the solution of the IVP \eqref{5kdvbbm}. This sort of problem is widely studied in the literature, for example see \cite{ILP-15, ILP-A1, ILP-A2, LPS-A1} and references therein. In most of these works, wide class of dispersive equations including the Korteweg-de Vries (KdV), Benjamin-Ono (BO), Benjamin-Bona-Mahony (BBM) equations is considered.

In the second part, we plan to study the well-posedness of the IVP \eqref{5kdvbbm} in the spaces of analytic functions, the so called Gevrey class of functions, and the  evolution of radius of analyticity of the solution. 

The  higher order water wave model (\ref{5kdvbbm}) describing the unidirectional propagation of water waves was recently introduced  by Bona et al. \cite{BCPS1} by using the second order approximation in the two-way model, the so-called $abcd-$system derived in \cite{BCS1, BCS2}. In the literature, this model  is also known as  the fifth order KdV-BBM type equation. The IVP \eqref{5kdvbbm} posed on the spatial domain $\R$ was studied by the authors in \cite{BCPS1} considering initial data in $H^s(\R)$ and proving the local well-posedness for $s\geq 1$. More precisely, the local result proved in \cite{BCPS1} is the following.

\begin{TA}\label{Th-A} 
 Assume $\gamma_1, \delta_1 >0$.  
For any $s\geq 1$ and for given  $\eta_0\in H^s(\R)$, there exist a time $T =T(\|\eta_0\|_{H^s})=\dfrac{c_s}{\|\eta_0\|_{H^s}(1+\|\eta_0\|_{H^s})}>0$
 and a unique function  $\eta \in C([0,T];H^s):=X^s_T$ which 
 is a solution of the IVP \eqref{5kdvbbm}, posed with initial data $\eta_0$.  The solution $\eta$ 
varies continuously in $C([0,T];H^s)$ as $\eta_0$ varies in $H^s(\R)$.
\end{TA}

 When the parameter $\gamma$ satisfies $\gamma = \frac7{48}$, the model  \eqref{5kdvbbm}  posed on $\R$ possesses hamiltonian structure  and the flow satisfies
\begin{equation}\label{energy}
 E(\eta(\cdot,t)):=  \int_{\mathbb{R}} \eta^2 + \gamma_1 (\eta_x)^2+\delta_1(\eta_{xx})^2\, dx= E(\eta_0).
\end{equation} 
We note that, this conservation law holds in the periodic case as well.

The energy conservation \eqref{energy} was used in \cite{BCPS1} to prove the global well-posedness for  data in $H^s(\R)$, $s\geq 2$. While, for data with  regularity $ \frac32\leq s< 2$, {\em splitting to high-low frequency parts} technique was used  by the authors in \cite{BCPS1} to get the global well-posedness result. This global well-posedness result was further improved in \cite{CP} for initial data with Sobolev regularity $s\geq 1$. Furthermore, the authors in \cite{CP} showed that the well-posedness result is sharp by proving that the mapping data-solution fails to be continuous at the origin whenever $s<1$. For similar results in the periodic case we refer to \cite{CPP-1} and for the problem posed on the half line we refer to \cite{hongqiu}.

As mentioned earlier, the first main interest of this work is to study the propagation of regularities in the solution of the IVP \eqref{5kdvbbm}. For motivation to this part of the work we mention the recent works \cite{LPS-A1} where the authors deal with a class of nonlinear dispersive equations that include the generalised KdV (gKdV), generalised Benjamin-Ono (gBO), Kadomtsev-Petviashvili (KP) and Benjamin-Bona-Mahony (BBM) equations among others. For other related works in this direction we refer to \cite{ILP-15}, \cite{ILP-A1}, \cite{ILP-A2} and references therein.

The main result regarding the propagation of regularities in the solution of the IVP \eqref{5kdvbbm} is stated as follows.

\begin{theorem}\label{Th-Propagation}
 Let $\gamma_1>0$, $0<\delta_1<\dfrac{\gamma_1^2}4$ and  $\eta_0\in H^s(\R)$ with $s\geq 1$. If for some $k\in \Z^+\cup \{0\}$,  $\theta\in [0, 1]$ and $\Omega\subseteq \R$ open
$${\eta_0}\big|_{\Omega}\in C^{k+\theta}$$
then the corresponding solution $\eta \in X^s_T$ of the  IVP \eqref{5kdvbbm} given by Theorem {\bf A} satisfies that
$${\eta(\cdot, t)}\big|_{\Omega}\in C^{k+\theta},\;\; \forall\; t\in[-T,T].$$
Moreover,
$$\eta, \; \eta_t\in C\big([-T, T]: C^{k+\theta}(\Omega)\big).$$
\end{theorem}

From Theorem \ref{Th-Propagation} we can infer that, in the   function space setting $C^{k+\theta}$,  in the interval $[0, T]$ of local existence,  no singularities can appear or disappear in the solution $\eta(\cdot, t)$ of the IVP \eqref{5kdvbbm}. Furthermore, as a consequence of the Theorem \ref{Th-Propagation} we have the following result.

\begin{corollary}\label{corl-1} Let $\gamma_1>0$ , $0<\delta_1<\dfrac{\gamma_1^2}4$ and  $\eta_0\in H^s(\R)$ with $s\geq 1$. If for $x_0 \in (a,b)$, $k\in \Z^+\cup \{0\}$ and $\theta\in [0, 1]$ we have
 $$\eta_0\big|_{(a,x_0)}, \,\, \eta_0\big|_{(x_0, b)} \in C^{k+\theta}, \quad \textrm{and} \quad \eta_0\big|_{(a, b)} \notin C^{k+\theta}$$
then the corresponding solution $\eta \in X^s_T$ of the  IVP \eqref{5kdvbbm} given by Theorem {\bf A} satisfies that
$${\eta(\cdot, t)}\big|_{(a,x_0)}, \,\, {\eta(\cdot, t)}\big|_{(x_0,b)} \in C^{k+\theta}, \quad \textrm{and} \quad {\eta(\cdot, t)}\big|_{(a,b)} \notin C^{k+\theta}.$$
\end{corollary}
\begin{proof}
We prove this result using a contradiction argument. If possible, suppose that ${\eta(\cdot, t)}\big|_{(a,b)} \in C^{k+\theta}$. Then we have $\eta(x,t+t')\in C^{k+\theta}$  for all $t'\in [-T,T]$. In particular, if $t'=-t$ we would have that $\eta_0\big|_{(a, b)} \in C^{k+\theta}$ resulting in a contradiction.
%where $\tilde{\eta}$ satisfies the fifth order KdV-BBM type equation \eqref{5kdvbbm}.

\end{proof}

The second main interest of this work is to find solutions $\eta(x,t)$ of the  IVP \eqref{5kdvbbm} with real-analytic initial data $\eta_0$ which admit extension as an analytic function to a complex strip $S_{\sigma_0}:=\big\{x+iy:|y|<\sigma_0\big\}$, for some $\sigma_0>0$ at least for a short time. Analytic Gevrey class introduced by Foias and Temam \cite{FT} is a suitable function space for this purpose. After getting this result, a natural question one may ask is whether this property holds globally in time, but with a possibly smaller radius of analyticity $\sigma(t)>0$. In other words, is the solution  $\eta(x,t)$ of the  IVP \eqref{5kdvbbm} with real-analytic initial data $\eta_0$ is analytic in $S_{\sigma(t)}$ for all $t$ and what is the lower-bound of $\sigma(t)$? This question will also be addressed in the second part of this work.

 An early work in this direction is due to Kato and Masuda \cite{KM}. They considered a large class of evolution equations and developed a general method to obtain spatial analyticity of the solution. In particular, the class considered in \cite{KM} contains the KdV equation. Further development  in this field can be found in the works of Hayashi \cite{H}, Hayashi and Ozawa \cite{HO},  de Bouard, Hayashi and Kato \cite{BHK}, Kato and Ozawa \cite{KO}, Bona and Gruji\'c \cite{BG-1}, Bona, Gruji\'c and Kalisch \cite{BGK-1, BGK-2}, Gruji\'c and Kalisch \cite{GK-1, GK-2}, Zhang \cite{Z1, Z2} and references there in. In recent literature, many authors have devoted much effort to get analytic solutions to several evolution equations, see for example \cite{BHG-17, HP-20, HP-18, HPS-17, SS-19, SS-18, ST-17, SdS-15} and references therein.

 Now, we  introduce some notations and define function spaces in which this work will be developed.  Throughout this work we use $C$ to denote a constant that may vary from one line to the next.
 
 The Fourier transform of a function $f$ is defined by
\begin{equation*}
\hat{f}(\xi) =
\frac{1}{\sqrt{(2\pi)}}\int_{\mathbb{R}}e^{-ix\xi}f(x)\,dx,
\end{equation*}
whose inverse transform is given by 
\begin{equation*}
f(x) =
\frac{1}{\sqrt{(2\pi)}}\int_{\mathbb{R}}e^{ix\xi}\hat f(\xi)\,d\xi.
\end{equation*}
We define a Fourier multiplier operator $J$ by
$$
\widehat{J^sf}(\xi) =\langle\xi\rangle^s \hat{f}(\xi).
$$
 For given $s\in \R$, we define the usual $L^2$-based Sobolev space  $H^s$ denotes of order $s$ with norm
$$
\|f\|_{H^s}^2 = \int_{\R}\langle\xi\rangle^{2s}|\hat f(\xi)|^2\,d\xi= \|J^sf\|_{L^2(\R)}^2=:\|J^sf\|^2,
$$
where $\langle\cdot\rangle = 1+|\cdot|$.

For $\sigma >0$ and $s\in \R$, the analytic Gevrey class $G^{\sigma, s}$ is defined as  the subspace of $L^2(\R)$ with norm,
\begin{equation*}\label{def-G1}
 \|f\|_{G^{\sigma, s}}^2 = \int_{\R} \langle \xi\rangle^{2s}
 e^{2\sigma\langle\xi\rangle}|\hat{f}(\xi)|^2d\xi.
\end{equation*}
The Gevrey norm of order $(\sigma, s)$ can be written in terms of the operator $J^s$ as 
\begin{equation}\label{def-G2}
\|f\|_{G^{\sigma, s}} =\|J^se^{\sigma J}f\|_{L^2(\R)}:=\|J^se^{\sigma J}f\|.
\end{equation}
To make the notation more compact we define $ J^{s, \sigma}\eta:=J^s e^{\sigma J} \eta$ so that the Gevrey norm can be expressed as  
\begin{equation*}\label{def-G3}\|\eta\|_{G^{\sigma,s}} =\|J^{s, \sigma}\eta\|_{L^2(\R)}=:\|J^{s, \sigma}\eta\|.
\end{equation*}

Note that, a function in the Gevrey class $G^{\sigma, s}$ is a restriction to the real axis of a function analytic on a symmetric strip of width $2\sigma$. Hence, our interest is to prove the well-posedness result for the IVP \eqref{5kdvbbm} for given data in $G^{\sigma, s}$ for appropriate $s$ and analyse how $\sigma =\sigma(t)$ evolves  in time.

%%%%%%%%%%%%%%%%

Now we are in position to state the main results of the second part of this work. The first result deals with  the local existence of the IVP \eqref{5kdvbbm}  for given data in the usual Gevrey space $G^{\sigma, s}(\R)$ and reads as follows.

\begin{theorem}\label{mainTh1}
Let $s\geq 1$, $\sigma>0$ and $\eta_0\in G^{\sigma, s}(\R)$ be given. Then there exist a time $T =T(\|\eta_0\|_{ G^{\sigma, s}}>0$ and $\eta\in C([0, T];  G^{\sigma, s})$  satisfying the  IVP \eqref{5kdvbbm}.

\end{theorem}

The next main result deals with the evolution of the radius of analyticity in time. More precisely,  we have  the following global well-posedness result.

\begin{theorem}
\label{global}
Let $\sigma:=\sigma(t)>0$ and  $\eta_0\in G^{\sigma, 2}(\R)$. Then, the solution $\eta$ of the IVP \eqref{5kdvbbm} with initial value $\eta_0$ remains analytic for all positive times. More precisely $u \in C(0,T; G^{\sigma, 2}(\R))$ for all $T>0$ where a lower bound  of the radius of analyticity $\sigma(t)$ is given by (see \eqref{3x11-m11})
$$
\sigma_0 \exp \{-(\| \eta_0\|_{G^{\sigma, 2}}+2\| \eta_0\|_{G^{\sigma, 2}}^2)t-\frac32 t^{3/2}(\|\eta_0\|_{H^2}^{3/2}+\|\eta_0\|_{H^2}^2)-t^2(\|\eta_0\|_{H^2}^{3/2}+\|\eta_0\|_{H^2}^2)^2 \}.
$$
%where $\mathcal{X}_0:=\| \eta_0\|_{G^{\sigma, 2}}$ and $\mathcal{Y}_0:=\|\eta_0\|_{H^2}^{3/2}+\|\eta_0\|_{H^2}^2$.
and an upper bound is given by 
$$
C \sigma_0 \exp\{ -\|   \eta_0 \|_{H^2}^2 t\}.
$$
Moreover
\begin{equation}\label{3x11-m9}
\begin{split}
\| \eta(t)\|_{G^{\sigma, 2}} \leq \| \eta_0\|_{G^{\sigma, 2}}+Ct^{1/2}(\|\eta_0\|_{H^2}^{3/2}+\|\eta_0\|_{H^2}^2).
\end{split}
\end{equation}
\end{theorem}

The local well-posedness will be established using multilinear estimates combined with a contraction mapping argument.  The global well-posedness
in the spaces $H^s$ with $s \geq 2$ will be obtained via energy-type arguments together with the local theory.   

As in the continuous case, with some restriction on the coefficients of the equation, we can also prove the  global well-posedness result in the periodic case too, i.e.,  for given data  $\eta_0\in G^{\sigma, 2}(\T)$.

Before finishing this section we record the following estimates that will be used in sequel.

\begin{remark}
Observe that if $\eta(x,t)$ is a solution of the IVP \eqref{5kdvbbm}, $c_1=\min \{\gamma_1, \delta_1\}$ and $C_1=\max\{\gamma_1, \delta_1\}$, then
\begin{equation}\label{consH2-1}
c_1\|\eta(\cdot,t)\|_{H^2}^2 \le E(\eta(\cdot,t))=E(\eta_0)\leq C_1 \|\eta(\cdot,t)\|_{H^2}^2.
\end{equation}

Also, it is clear that 
$
c_1\|\eta(\cdot,t)\|_{H^2}^2 \leq C_1 \|\eta_0\|_{H^2}^2
$
and 
$
c_1 \|\eta_0\|_{H^2}^2 \leq C_1\|\eta(\cdot,t)\|_{H^2}^2
$. Therefore
\begin{equation}\label{consH2}
\frac{c_1}{C_1}\|\eta_0\|_{H^2}^2 \leq \|\eta(\cdot,t)\|_{H^2}^2 \leq \frac{C_1}{c_1}\|\eta_0\|_{H^2}^2.
\end{equation}
\end{remark}

%%%%%%%%%%%%%%%%%%%%%%%%%%%%%%%%%%%%%%%%%%%%%%%%%%%%%%%%%%%%%%%%%%%%%%%%%%%%%%%%%%%
%%%%%%%%%%%%%%%%%%%%%%%%%%%%%%%%%%%%%%%%%%%%%%%%%%%%%%%%%%%%%%%%%%%%%%%%%%%%%%%%%%%

\section{Propagation of Regularities}
This part of our work is inspired by the recent work of Linares et. al \cite{LPS-A1} where the authors studied the  propagation of the regularities in the solutions to the BBM equation.

As discussed in the introduction, for the well-posedness results to the IVP \eqref{5kdvbbm} we refer to \cite{BCPS1, CP}  where the authors proved the sharp local well-possednes and the global well-posedness in $H^1(\R)$. In particular if the initial data $\eta_0$ is in $H^1(\R)$, then for any $T>0$ the solution of the IVP \eqref{5kdvbbm} is in $C([0,T]; H^1(\R))=:X^1_T$. 

In what follows we provide proof of Theorem \ref{Th-Propagation}. We will closely  follow the technique introduced in \cite{LPS-A1} in the context of the third order BBM equation.

\begin{proof}[Proof of Theorem \ref{Th-Propagation}]
For technical reasons we make a change of variables given by $\eta(x, t) \equiv \eta (x-\dfrac{\delta_2}{\delta_1} t, t)$. With this change of variables the equation \eqref{5kdvbbm} transforms to
\begin{equation}\label{5kdvbbmB}
\begin{cases}
\eta_t+\delta_3 \eta_x-\gamma_1 \eta_{xxt}+\gamma_3\eta_{xxx}+\delta_1 \eta_{xxxxt}+\dfrac{3}{2}\eta \eta_x+\gamma (\eta^2)_{xxx}-\dfrac{7}{48}(\eta_x^2)_x-\dfrac{1}{8}(\eta^3)_x=0,   \\
\eta(x,0)=\eta_0(x),
\end{cases} 
\end{equation}
where $\delta_3=1-\dfrac{\delta_2}{\delta_1}$ and $\gamma_3=\gamma_2+\gamma_1\dfrac{\delta_2}{\delta_1}$. This sort of change of variables was used in \cite{BCG} and \cite{CPP-1}. This transformation  eliminates the fifth order term $\eta_{xxxxx}$ and only alters the coefficients of the terms $\eta_{x}$ and $\eta_{xxx}$.

We provide detailed proof considering $\eta_0\in H^1(\R)$. For $s>1$ the proof follows analogously. From Theorem {\bf A},  there exist $T= T(\|\eta_0\|_{H^1})>0$  and a unique solution $\eta= \eta(x,t)$ of the IVP \eqref{5kdvbbmB} such that
\begin{equation*}\label{eqR1}
\eta\in C([0, T] ; H^1(\R)).
\end{equation*}

We define  a multiplication operator $\mathcal{J}$ by
$$\widehat{ \mathcal{J}f}(\xi)= \varphi(\xi)\widehat{f}(\xi),$$
where 
\begin{equation*}\label{var-phi}
\varphi(\xi) = 1+\gamma_1\xi^2+\delta_1\xi^4.
\end{equation*}

Using the operator $\mathcal{J}$, the equation \eqref{5kdvbbmB}, can be written as
\begin{equation}\label{Reg1}
 \mathcal{J}\eta_t+\delta_3\eta_x+\gamma_3\eta_{xxx}+\frac{3}{2}\eta \eta_x+\gamma (\eta^2)_{xxx}-\frac{7}{48}(\eta_x^2)_x-\frac{1}{8}(\eta^3)_x=0,
\end{equation}
and consequently
\begin{equation}\label{Reg1}
\eta_t=\mathcal{J}^{-1}\partial_x\Big(-\delta_3\eta-\frac{3}{4}\eta^2+\frac{1}{8}\eta^3+   \frac{7}{48} \eta_x^2 \Big)-\mathcal{J}^{-1}\partial_x^3\big(\gamma_3\eta+\gamma \eta^2 \big).
\end{equation}

In what follows, we will work on the  operators $\mathcal{J}^{-1}\partial_x$ and $\mathcal{J}^{-1}\partial_x^3 $.  For technical reason we suppose that $\gamma_1^2-4\delta_1 \geq 0$ and write the operator $\mathcal{J}$ as
$$
\mathcal{J}=\mathcal{J}_1\mathcal{J}_2,  
$$
where 
$$
\widehat{ \mathcal{J}_1f}(\xi)= (1+a_1^2\xi^2)\widehat{f}(\xi), \quad \widehat{ \mathcal{J}_2f}(\xi)= (1+a_2^2\xi^2)\widehat{f}(\xi),
$$
with $a_1^2 a_2^2=\delta_1$, $a_1^2+ a_2^2=\gamma_1$, $a_1>0$, $a_2>0$. 
\vspace{5mm}

\noindent
$\underline{\textrm{\bf The operator }\, \mathcal{J}^{-1}\partial_x}${\bf :}  To analyse this operator we recall  the following Fourier transform
\begin{equation}\label{TFour}
\mathcal{F}(e^{-|y|/a})(\xi)=\dfrac{a\sqrt{2/\pi}}{1+a^2\xi^2}, \quad \textrm{and} \quad \mathcal{F}(e^{-|y|/a} \sgn(y))(\xi)=\dfrac{-i \sqrt{2/\pi} \,a^2}{1+a^2\xi^2}\xi.
\end{equation}
Now, using the definitions of the operators $\mathcal{J}_1$ and $\mathcal{J}_2$, and the relations \eqref{TFour} 
we have
\begin{equation}\label{op1}
\mathcal{J}^{-1}\partial_x f(x)\sim [(e^{-|y|/a_1} \sgn(y) )\ast(e^{-|y|/a_2} )\ast f(y)](x).
\end{equation}
\vspace{2mm}

\noindent
$\underline{\textrm{\bf The operator }\, \mathcal{J}^{-1}\partial_x^3}${\bf :} To analyse this operator we consider two different cases:\\

\begin{itemize}
\item{\bf Case (i)}: $\underline{a_1^2\neq a_2^2}$.
 In this case, using partial fractions, we observe that
$$
\dfrac{\xi^3}{(1+a_1^2\xi^2)(1+a_2^2\xi^2)}=\dfrac{1}{a_2^2-a_1^2}\left(  \dfrac{\xi}{1+a_1^2\xi^2}- \dfrac{\xi}{1+a_2^2\xi^2}\right).
$$
Consequently we can write
\begin{equation}\label{op2}
\mathcal{J}^{-1}\partial_x^3 f(x)\sim  [(e^{-|y|/a_1} \sgn(y) )\ast f(y)](x)- [(e^{-|y|/a_2} \sgn(y) )\ast f(y)](x).
\end{equation}

\item {\bf  Case (ii)}: $\underline{a_1^2=a_2^2}$. In this case,   using partial fractions
$$
\dfrac{\xi^3}{(1+a_1^2\xi^2)(1+a_2^2\xi^2)}=\dfrac{1}{a_1^2}\left(  \dfrac{\xi}{1+a_1^2\xi^2}- \dfrac{\xi}{(1+a_1^2\xi^2)^2}\right).
$$
Therefore, similarly to \eqref{op1}, one has
\begin{equation}\label{op3}
\mathcal{J}^{-1}\partial_x^3 f(x)\sim  [(e^{-|y|/a_1} \sgn(y) )\ast f(y)](x)- [(e^{-|y|/a_1} \sgn(y) )\ast(e^{-|y|/a_1})\ast f(y)](x).
\end{equation}

\end{itemize}

\vspace{5mm}
Since $\eta \in C([0, T ] : H^1(\R))$, we have $-\frac{3}{4}\eta^2+\frac{1}{8}\eta^3+   \frac{7}{48} \eta_x^2 \in C([0, T ] : L^1(\R))$. Therefore, in the light of \eqref{op1},
\begin{equation}\label{RT-1}
\mathcal{J}^{-1}\partial_x\left(-\delta_3\eta-\frac{3}{4}\eta^2+\frac{1}{8}\eta^3+   \frac{7}{48} \eta_x^2 \right)   \in C([0, T ] : C_b(\R)),
\end{equation}
where $C_b(\R) = C(\R) \cap L^\infty(\R)$.\\

Also, $\gamma_3\eta+\gamma \eta^2 \in C([0, T ] : H^1(\R))$. So, using Sobolev embedding theorem we have that
\begin{equation}\label{RT-2}
\mathcal{J}^{-1}\partial_x^3\left(\gamma_3\eta+\gamma \eta^2 \right) \in C([0, T ] : H^2(\R)) \hookrightarrow C([0, T ] : C^1_\infty(\R)),
\end{equation}
where
$$
C^1_\infty(\R))=\{u: \R \to \R; \quad \textrm{u is continuously diferentiable with} \quad \lim_{|x| \to \infty} u(x)=0\}.
$$

Integrating \eqref{Reg1} in time variable, we get
\begin{equation}\label{regul1}
\begin{split}
\eta(x,t)&=\eta_0-\int_0^t \mathcal{J}^{-1}\partial_x\left(\delta_3\eta+\frac34\eta^2-\frac18\eta^3-\frac7{48}\eta_x^2 \right)(x,t')dt'\\
&\qquad-\int_0^t\mathcal{J}^{-1}\partial_x^3\left(\gamma_3\eta+\gamma \eta^2 \right)(x,t') dt'\\
&=:\eta_0-w(x,t).
\end{split}
\end{equation}

Note that, from \eqref{RT-1} and \eqref{RT-2}, one has
\begin{equation}\label{RT-3}
w\in C([0, T]; C_b(\R)).
\end{equation}
Hence, if  $\eta_0\big|_{\Omega} \in C(\Omega)$ for some open $\Omega\subset \R$, we can conclude that
\begin{equation}
\label{RT-4}
\eta\big|_{\Omega\times[0, T]} \in C([0, T]; C(\Omega)).
\end{equation}

Now, differentiating $w(x,t)$ with respect to the space variable $x$, we obtain
\begin{equation}\label{regul1.1}
\begin{split}
\partial_xw(x)&=\int_0^t \mathcal{J}^{-1}\partial_x^2\left(\delta_3\eta+\frac34\eta^2-\frac18\eta^3-\frac7{48}\eta_x^2 \right)(x,t')d t'\\
&\qquad+\int_0^t\mathcal{J}^{-1}\partial_x^4\left(\gamma_3\eta+\gamma \eta^2 \right)(x,t') dt'.
\end{split}
\end{equation}

Observe that the definition of operator $\mathcal{J}$ implies  
$(I-\gamma_1\partial_x^2+\delta_1 \partial_x^4)\mathcal{J}^{-1}=I.$ Consequently
\begin{equation}\label{regul2}
\begin{split}
\delta_1 \partial_x^4\mathcal{J}^{-1}=I-\mathcal{J}^{-1}+\gamma_1\partial_x^2\mathcal{J}^{-1}.
\end{split}
\end{equation}

Hence, using \eqref{regul2} in   \eqref{regul1.1} yields
\begin{equation}\label{regul3}
\begin{split}
\partial_xw(x,t)&=\dfrac{1}{\delta_1}\int_0^t \left(\gamma_3\eta+\gamma \eta^2 \right)(x,t') d t'-\dfrac{1}{\delta_1} \int_0^t\mathcal{J}^{-1}\left(\gamma_3\eta+\gamma \eta^2 \right)(x,t') dt'\\
&\qquad  +\dfrac{1}{\delta_1}\int_0^t \mathcal{J}^{-1}\partial_x^2\left((\delta_3+\gamma_3)\eta+\big(\frac34+\gamma\big)\eta^2-\frac18\eta^3-\frac7{48} \eta_x^2 \right)(x,t')dt'.
\end{split}
\end{equation}

From \eqref{RT-4}, we have
\begin{equation}
\label{RT-5}
\gamma_3\eta+\gamma \eta^2\in  C([0, T]; C(\Omega)).
\end{equation}

Similarly as above, using the Fourier transforms \eqref{TFour}, we have
\begin{equation}\label{op4}
\mathcal{J}^{-1} f(x)\sim [(e^{-|y|/a_1}  )\ast(e^{-|y|/a_2} )\ast f(y)](x).
\end{equation}
and
\begin{equation}\label{op5}
\mathcal{J}^{-1}\partial_x^2 f(x)\sim [(e^{-|y|/a_1} \sgn(y) )\ast(e^{-|y|/a_2}\sgn(y) )\ast f(y)](x).
\end{equation}
Using \eqref{op4}, \eqref{op5} and an analogous argument as in \eqref{RT-1} and \eqref{RT-2}, one gets
\begin{equation}
\label{RT-6}
\mathcal{J}^{-1}\!\left(\gamma_3\eta+\gamma \eta^2\right),\; \; \mathcal{J}^{-1}\partial_x^2\left(\!\!(\delta_3+\gamma_3)\eta+(\frac34+\gamma)\eta^2-\frac18\eta^3-\frac7{48} \eta_x^2 \right) \in C([0, T]; C_b(\R)).
\end{equation}

Now, from \eqref{regul3}, \eqref{RT-5} and \eqref{RT-6} we can infer that
\begin{equation}
\label{RT-7}
w\big|_{\Omega\times[0, T]} \in C([0, T]; C^1(\Omega)),
\end{equation}
whenever $\eta_0\in C(\Omega)$ for some open $\Omega\subset \R$.

Therefore, for any $\theta\in (0, 1]$, if  $\eta_0\big|_{\Omega} \in C^{\theta}(\Omega)$ for some open $\Omega\subset \R$, we  have from \eqref{regul1} that
\begin{equation}
\label{RT-7}
\eta\big|_{\Omega\times[0, T]} \in C([0, T]; C^{\theta}(\Omega)).
\end{equation}

Now, repeating the above procedure for  $\eta_0\big|_{\Omega} \in C^{1}(\Omega)$ for some open $\Omega\subset \R$, we obtain
\begin{equation}
\label{RT-8}
\eta\big|_{\Omega\times[0, T]} \in C([0, T]; C^{1}(\Omega)),
\end{equation}
and from \eqref{regul1.1}
\begin{equation}
\label{RT-9}
w \in C([0, T]; C^{2}(\Omega)).
\end{equation}

Hence, for any $\theta\in (0, 1]$,  if  $\eta_0\big|_{\Omega} \in C^{1+\theta}(\Omega)$ for some open $\Omega\subset \R$, we can conclude from \eqref{Reg1}, \eqref{RT-1}, \eqref{RT-2} and  \eqref{regul1} that
\begin{equation*}
\label{RT-10}
\eta\big|_{\Omega\times[0, T]} \in C^1([0, T]; C^{1+\theta}(\Omega)).
\end{equation*}

Finally, repeating the above procedure, one can conclude the proof of the Theorem.
\end{proof}
%%%%%%%%%%%%%%%%%%%%%%%%%%%%%%%%%%%%%%%%%%%%%%%%%%%%%%%%%%%%%%%%%%%%%%%%%%%%%%%%%%%

\setcounter{equation}{0}
\section{Local Well-posedness Theory in $G^{\sigma, s}$, $s\geq 1$}\label{sec-3}
%%%%%%%%%%%%%%%%%%%%%%%%%%%%%%%%%%%%%%%%%%%%%%%%%%%%%%%%%%%%%%%%%%%%%%%%%%%%%%%%%%%%%%%%%%%%%%%%%%%%%%%%%%%%%%

%%%%%%%%%%%%%%%%%%%%%%%%%%%%%%%%%%%%%%%%%%%%%%%%%%%

%%%%%%%%%%%%%%%%%%%%%%%%%%%%%%%%%%%%%%%%%%%%%%%%%%%%%%%%%%

In this section we focus on the local well-posedness issues of the IVP (\ref{5kdvbbm})
for given data  $\eta_0\in {G^{\sigma, s}}$, $s\geq 1$.  We start writing the IVP (\ref{5kdvbbm}) in an equivalent 
integral equation format. Taking the Fourier transform in the first equation in (\ref{5kdvbbm}) with respect to the spatial variable and organizing the terms, we get
\begin{equation}\label{eq1.8}
\Big(1+\gamma_1\xi^2+\delta_1 \xi^4\Big)i\widehat\eta_t =\xi(1-\gamma_2\xi^2+\delta_2\xi^4)\widehat\eta+\frac14(3\xi-4\gamma\xi^3)\widehat{\eta^2}  -\frac18 \xi \widehat{\eta^3} -\frac7{48}\xi \widehat{\eta_x^2}.
\end{equation}

The fourth-order polynomial 
\begin{equation}\label{varphi}
 \varphi(\xi) := 1 + \gamma_1\xi^2+\delta_1\xi^4>0,
\end{equation}
is strictly positive because  $\gamma_1$, and $\delta_1$ are taken to be positive. 
 
Now, we define the  Fourier multiplier operators $\phi(\partial_x)$, $\psi(\partial_x)$ and $\tau(\partial_x)$ as follows
\begin{equation}\label{phi-D}
\widehat{\phi(\partial_x)f}(\xi):=\phi(\xi)\widehat{f}(\xi), \qquad \widehat{\psi(\partial_x)f}(\xi):=\psi(\xi)\widehat{f}(\xi) \;\;\; {\rm and }\;\;\; \widehat{\tau(\partial_x)f}(\xi):=\tau(\xi)\widehat{f}(\xi), 
\end{equation}
 where the symbols are given by
\begin{equation*}
  \phi(\xi)=\frac{\xi(1-\gamma_2\xi^2+\delta_2\xi^4)}{\varphi(\xi)}, \quad \psi(\xi)=\frac{\xi}{\varphi(\xi)} \quad  {\rm and} \quad \tau(\xi)=\frac{3\xi-4\gamma\xi^3}{4\varphi(\xi)}.
\end{equation*}

With this notation, the IVP (\ref{5kdvbbm}) can be written in the form
\begin{equation}\label{eq1.9}
\begin{cases}
i\eta_t = \phi(\partial_x)\eta + \tau (\partial_x)\eta^2 - \frac18\psi(\partial_x)\eta^3  -\frac7{48}\psi(\partial_x)\eta_x^2\, ,\\
 \eta(x,0) = \eta_0(x).
 \end{cases}
\end{equation}
Consider first the following linear IVP associated to \eqref{eq1.9}
\begin{equation}\label{eq1.10}
\begin{cases}
i\eta_t = \phi(\partial_x)\eta,\\
\eta(x,0) = \eta_0(x),
\end{cases}
\end{equation}
whose solution is given  by $\eta(t) = S(t)\eta_0$, where $\widehat{S(t)\eta_0} = e^{-i\phi(\xi)t}\widehat{\eta_0}$ is defined via its Fourier transform.
Clearly, $S(t)$ is a unitary operator on $H^s$ and $G^{\sigma, s}$ for any $s \in \R$, so that
\begin{equation}\label{eq1.11}
\|S(t)\eta_0\|_{H^s} = \|\eta_0\|_{H^s},\qquad \textrm{and}\qquad \|S(t)\eta_0\|_{G^{\sigma, s}} = \|\eta_0\|_{G^{\sigma,s}} 
\end{equation}
for all $t > 0$.
Duhamel's formula allows us to write the IVP  (\ref{eq1.9}) in the equivalent integral equation form,
\begin{equation}\label{eq1.12}
\eta(x,t) = S(t)\eta_0 -i\int_0^tS(t-t')\Big(\tau(\partial_x)\eta^2 - \frac18 \psi(\partial_x)\eta^3 -\frac7{48}\psi (\partial_x)\eta_x^2\Big)(x, t') dt'.
\end{equation}

In what follows, a short-time solution of (\ref{eq1.12}) will be obtained via the contraction mapping principle in the space $C([0,T];G^{\sigma, s})$.  This will provide a 
proof of Theorem \ref{mainTh1}.

%%%%%%%%%%%%%%%%%%%%%%%%%%%%%%%%%%%%%%%%%%%%%%%%%%%%%%%%%%%%%%%%%%
\subsubsection{Multilinear Estimates}
%%%%%%%%%%%%%%%%%%%%%%%%%%%%%%%%%%%%%%%%%%%%%%%%%%%%%%%%%%%%%%%%%

 Various multilinear estimates are now established that will be useful in the proof of the local well-posedness result.
  First, we record the  following $G^{\sigma, s}$ version of the  ``sharp" bilinear estimate obtained in \cite{BT}.

\begin{lemma}\label{BT1}
 For $s \ge 0$, there is a constant $C = C_s$ for which
\begin{equation}\label{bt}
\|\omega(\partial_x) (u v)\|_{G^{\sigma, s}} \le C\|u\|_{G^{\sigma,s}}\|v\|_{G^{\sigma,s}}
\end{equation}
 where $\omega(\partial_x)$  is the Fourier multiplier operator 
with symbol
\begin{equation*} \label{btx0}
\omega(\xi) \, = \, \frac{|\xi|}{1 + \xi^2}.
\end{equation*} 
\end{lemma}

\begin{proof}
Using  the definition of the $G^{\sigma, s}$-norm from \eqref{def-G2}, one can obtain
\begin{equation}\label{bil-m1}
\begin{split}
\|\omega(\partial_x) (uv)\|_{G^{\sigma, s}} ^2&= \|\langle\xi\rangle^s e^{\sigma\langle\xi\rangle}\omega(\xi)\widehat{u}*\widehat{v}(\xi)\|_{L^2}^2\\
&=\int_{\R}\langle\xi\rangle^{2s} e^{2\sigma\langle\xi\rangle}\frac{\xi^2}{(1+\xi^2)^2}\Big(\int_{\R}\widehat{u}(\xi-\xi_1)\widehat{v}(\xi_1)d\xi_1\Big)^2d\xi.
\end{split}
\end{equation}

Note that, for $s\geq 0$ one has $\langle\xi\rangle^s\leq \langle\xi-\xi_1\rangle^s\langle\xi_1\rangle^s$ and also $e^{\sigma\langle\xi\rangle}\leq e^{\sigma\langle\xi-\xi_1\rangle}e^{\sigma\langle\xi_1\rangle}$. Using these inequalities, the estimate \eqref{bil-m1} yields
\begin{equation}\label{bil-m2}
\begin{split}
\|\omega(\partial_x) (uv)\|_{G^{\sigma, s}} ^2&\leq
\int_{\R}\frac{\xi^2}{(1+\xi^2)^2}\Big(\int_{\R}\langle\xi-\xi_1\rangle^{2s} e^{2\sigma\langle\xi-\xi_1\rangle}\widehat{u}(\xi-\xi_1)\langle\xi_1\rangle^{2s} e^{2\sigma\langle\xi_1\rangle}\widehat{v}(\xi_1)d\xi_1\Big)^2d\xi.
\end{split}
\end{equation}

Now, using $\frac{\xi^2}{(1+\xi^2)^2}\leq\frac1{1+\xi^2}$, the Cauchy-Schwartz inequality and the definition of the $G^{\sigma, s}$-norm, we obtain from \eqref{bil-m2} that
\begin{equation}\label{bil-m3}
\begin{split}
\|\omega(\partial_x) (uv)\|_{G^{\sigma, s}} ^2&\leq
\int_{\R}\frac{1}{1+\xi^2}d\xi \:\;\|u\|_{G^{\sigma, s}}^2\|v\|_{G^{\sigma, s}}^2\\
&\leq C\|u\|_{G^{\sigma, s}}^2\|v\|_{G^{\sigma, s}}^2,
\end{split}
\end{equation}
and this completes the proof.
\end{proof}

It is worth noting that the counterexample in \cite{BT} can be adapted to  show that the inequality (\ref{bt}) fails for  $s<0$.

\begin{lemma}\label{Lema1}
For any $s \ge 0$ and $\sigma>0$, there is a constant $C = C_s$ such that the inequality
\begin{equation}\label{bilin-1}
\|\tau(\partial_x) \eta^2\|_{G^{\sigma, s}} \le C \| \eta\|_{G^{\sigma, s}} ^2
\end{equation}
holds, where the operator $\tau(\partial_x)$  is  defined in (\ref{phi-D}).
\end{lemma}

\begin{proof}
Since  $\delta_1>0$, one can easily verify that  $|\tau(\xi)| \leq C \omega(\xi)$ for some constant $C>0$.  Using this fact and the definition of the $G^{\sigma, s}$-norm and the estimate \eqref{bt} from Lemma \ref{BT1}, one can obtain
\begin{equation*}%\label{bil-m1}
\begin{split}
\|\tau(\partial_x) \eta^2\|_{G^{\sigma, s}} &\leq \|\langle\xi\rangle^s e^{\sigma\langle\xi\rangle}\tau(\xi)\widehat{\eta}*\widehat{\eta}(\xi)\|_{L^2}\\
&\leq \|\langle\xi\rangle^s e^{\sigma\langle\xi\rangle}\omega(\xi)\widehat{\eta}*\widehat{\eta}(\xi)\|_{L^2}^2\\
&\leq C\|\eta\|_{G^{\sigma, s}}^2,
\end{split}
\end{equation*}
as required.
\end{proof}

\begin{lemma}\label{P1}
For $s \ge \frac16$, there is a constant $C = C_s$ such that 
\begin{equation}\label{trilin-1}
\|\psi(\partial_x) \eta^3\|_{G^{\sigma, s}} \le C \| \eta\|_{G^{\sigma, s}} ^3.
\end{equation}
\end{lemma}
\begin{proof}
Consider first when $\frac16 \le s < \frac52$.   In this case, it appears that 
\begin{equation*}\label{x2}
\Big|(1+|\xi|)^s \,\psi(\xi)\Big|=\Big|\frac{ (1+|\xi|)^s \xi}{(1 +\gamma_1 \xi^2+\delta_1\xi^4)}\Big|
 \le C \frac{1}{(1+|\xi|)^{3-s}}.
\end{equation*}

 The last inequality implies that
\begin{equation}\label{x11}
\begin{split}
\|\psi(\partial_x) \eta^3\|_{G^{\sigma, s}} &= \|(1+|\xi|)^s \,\psi(\xi)e^{\sigma\langle\xi\rangle}\widehat{\eta^3}(\xi)\|_{L^2}\\
& \le C\left\|\frac{1}{(1+|\xi|)^{3-s}} e^{\sigma\langle\xi\rangle}\widehat{\eta^3}(\xi)\right\|_{L^2}\\ 
&\leq C \left\|\frac{1}{(1+|\xi|)^{3-s}}\right\|_{L^2}\| e^{\sigma\langle\xi\rangle}\widehat{\eta^3}(\xi)\|_{L^{\infty}}.
\end{split}
\end{equation}

Let $\widehat{f}(\xi):=e^{\sigma\langle\xi\rangle}\widehat{\eta}(\xi)$. Then using  $e^{\sigma\langle\xi\rangle}\leq e^{\sigma\langle\xi-\xi_1-\xi_2\rangle}e^{\sigma\langle\xi_1\rangle}e^{\sigma\langle\xi_2\rangle}$, we get
\begin{equation}\label{n-33}
e^{\sigma\langle\xi\rangle}\widehat{\eta^3}(\xi)\leq \int_{\R}e^{\sigma|\langle\xi-\xi_1-\xi_2\rangle}\widehat{\eta}(\xi-\xi_1-\xi_2)e^{\sigma\langle\xi_1\rangle}\widehat{\eta}(\xi_1)e^{\sigma\langle\xi_2\rangle} \widehat{\eta}(\xi_2)d\xi_1d\xi_2 = \widehat{f^3}(\xi).
\end{equation}

Using \eqref{n-33} and the fact that $\left\|\frac{1}{(1+|\xi|)^{3-s}}\right\|_{L^2}$ is bounded for $s<\frac52$, we obtain from \eqref{x11} that
\begin{equation}\label{x11-m}
\begin{split}
\|\psi(\partial_x) \eta^3\|_{G^{\sigma, s}} &\leq \| \widehat{f^3}(\xi)\|_{L^{\infty}} \leq C\|f\|_{L^3}^3.
\end{split}
\end{equation}

 From one dimensional Sobolev embedding, we have 
\begin{equation}\label{x3}
\|f\|_{L^{3}} \le C \|f\|_{H^\frac16} = C\|\eta\|_{G^{{\sigma},\frac16}}.
\end{equation}

Therefore, for   $\frac16 \leq s < \frac52$, we obtain from \eqref{x11-m} and \eqref{x3} that
\begin{equation}
\|\psi(\partial_x) \eta^3\|_{G^{\sigma, s}} \le  C\| \eta\|_{G^{\sigma, s}} ^3.
\end{equation}

For $s\geq \frac52$,  we observe that  $G^{\sigma,s} $ is a Banach algebra.  Also, note that $|\psi(\xi)| \leq C\frac{|\xi|}{1+\xi^2}$. So, using the same procedure as in Lemma \ref{Lema1}, we obtain
\begin{equation*}
\|\psi(\partial_x) (\eta \eta^2)\|_{G^{\sigma, s}} \le C\|\eta \|_{G^{\sigma, s}} \|\eta^2
\|_{G^{\sigma,s}}\le C\|\eta \|_{G^{\sigma, s}}^3,
\end{equation*}
as desired.
\end{proof}

\begin{lemma}\label{lema2.5}
For $s \ge 1$, the inequality 
\begin{equation}\label{Sharp1}
\|\psi(\partial_x) \eta_x^2\|_{G^{\sigma,s}} \le C \| \eta\|_{G^{\sigma,s}} ^2
\end{equation}
holds.
\end{lemma}
\begin{proof}
Observe that
$$
\psi(\xi) \leq C\omega(\xi) \frac1{1+ |\xi|}.
$$
The inequality (\ref{bt})  then allows the conclusion 
\begin{equation*}\label{x6}
\|\psi(\partial_x) \eta_x^2\|_{G^{\sigma,s}}  \leq C\| \omega(\partial_x) \eta_x^2\|_{G^{\sigma,s-1}} \le  C\|  \eta_x\|_{G^{\sigma,s-1}}\|  \eta_x\|_{G^{\sigma,s-1}}
\le  C\|  \eta\|_{G^{\sigma,s}}^2,
\end{equation*}
since  $s-1 \ge 0$.
\end{proof}
%\begin{remark}
%The sharp estimate (\ref{x6} proved in \cite{BT} and the relation (\ref{x5} implies  that the inequality (\ref{Sharp1} is also sharp.
%\end{remark}

In what follows, we use the estimates derived above  to provide a proof of the local well-posedness result in the $G^{\sigma,s}(\R)$ space whenever $s\geq 1$.
\begin{proof}[Proof of Theorem \ref{mainTh1}]
 Taking into account of the Duhamel's formula \eqref{eq1.12}, we define a mapping
\begin{equation}\label{eq3.42}
\Psi\eta(x,t) = S(t)\eta_0 -i\int_0^tS(t-t')\Big(\tau(D_x)\eta^2 - \frac14 \psi(\partial_x)\eta^3
 -\frac7{48} \psi(\partial_x)\eta_x^2\Big)(x, t') dt'.
\end{equation}
We show that the mapping $\Psi$ is a contraction on a closed ball $\mathcal{B}_r$ with radius $r > 0$ and center at the origin in $C([0,T];G^{\sigma,s})$.

From  \eqref{eq1.11}, we know that $S(t)$ is a unitary group in $G^{\sigma,s}(\R)$. Using this fact, we obtain
\begin{equation}\label{eq3.43}
\|\Psi\eta\|_{G^{\sigma,s}} \leq \|\eta_0\|_{G^{\sigma,s}} +CT\Big[\big{\|}\tau(\partial_x)\eta^2 - \frac18 \psi(\partial_x)\eta^3
 -\frac7{48}\psi(\partial_x)\eta_x^2\big{\|}_{C([0,T];G^{\sigma,s})}\Big].
\end{equation}
In view of the  inequalities (\ref{bilin-1}), (\ref{trilin-1}) and (\ref{Sharp1}), we obtain from \eqref{eq3.43} that
\begin{equation}\label{eq3.44}
\|\Psi\eta\|_{G^{\sigma,s}} \leq \|\eta_0\|_{G^{\sigma,s}} +CT\Big[\big{\|}\eta\big{\|}_{C([0,T];G^{\sigma,s})}^2 + \big{\|}\eta\big{\|}_{C([0,T];G^{\sigma,s})}^3 +\big{\|}\eta\big{\|}_{C([0,T];G^{\sigma,s})}^2\Big].
\end{equation}

Now, consider  $\eta\in \mathcal{B}_r$, then (\ref{eq3.44}) yields
\begin{equation*}\label{eq3.45}
\|\Psi\eta\|_{G^{\sigma,s}} \leq \|\eta_0\|_{G^{\sigma,s}} +CT\big[2r +r^2 \big]r.
\end{equation*}

If we choose $r= 2\|\eta_0\|_{H^s}$ and $T= \frac1{2Cr(2 + r) }$, then  $\|\Psi\eta\|_{G^{\sigma,s}} \leq r$, showing that $\Psi$ maps
 the closed ball $\mathcal{B}_r$ in $C([0,T];G^{\sigma,s})$ onto itself.   With the same choice of $r$ and $T$ and the same sort of 
estimates, one can easily show that $\Psi$ is a contraction on $\mathcal{B}_r$ with contraction constant equal to $\frac12$ as it happens. The rest of
the proof is standard so we omit the details.
\end{proof}

\begin{remark}\label{rm2.1}
The following points follow immediately from the proof of the Theorem \ref{mainTh1}:
\begin{enumerate}
\item The maximal existence time $T_s$ of the solution satisfies
\begin{equation}\label{r2.45}
T_s\geq \bar{T} = \frac1{8C_s\|\eta_0\|_{G^{\sigma,s}}(1+\|\eta_0\|_{G^{\sigma,s}})},
\end{equation}
where the constant $C_s$ depends only on $s$.
\item The solution cannot grow too much on the interval $[0,\bar T]$ since
\begin{equation}\label{r2.46}
\|\eta(\cdot,t)\|_{G^{\sigma,s}} \leq r =  2\|\eta_0\|_{G^{\sigma,s}}
\end{equation}
for  $t $ in this interval,  where $\bar{T}$ is as above in (\ref{r2.45}).
\end{enumerate}
\end{remark}

%%%%%%%%%%%%%%%%%%%%%%%%%%%%%%%%%%%%%%%%%%%%%%%%%%%%%%%%%%%%%%%%%%%%%%%%%%%%%%%%%%%

\setcounter{equation}{0}
\section{Evolution of Radius of Analyticity}\label{sec-3}
In this section we study the evolution of the radius of analyticity $\sigma(t)$ as $t$ grows. 
\begin{lemma}\label{lem31}
Let $r,s$ and $\sigma$ be non-negative numbers. Then there are absolute constants $c_1$ and $c_2$ such that for $u \in D(J^{r+s}e^{\sigma J})$,
$$
\|J^{s, \sigma}u\| \leq c_1\|J^{s}u\|+c_2\sigma^r \|J^{s+r, \sigma}u\|.
$$
\end{lemma}
\begin{proof}
See Lemma 9 in \cite{BG-1}.
\end{proof}
\begin{lemma}\label{lem32}
Let $s_1,s_2,s$  be such that $s_1 \leq s \leq s_2$ and $\sigma$ be non-negative number. Then 
$$
\|J^{s, \sigma}u\| \leq \|J^{s_1, \sigma}u\|^{\theta}\|J^{s_2, \sigma}u\|^{1-\theta},
$$
where $s =\theta s_1 + (1-\theta) s_2$.
\end{lemma}
\begin{proof}
This inequality is a consequence of the Holder's inequality. In fact
\begin{equation*}%\label{3x11-m}
\begin{split}
\|J^{s, \sigma}u\|^2&= \int_{\R} \langle \xi \rangle^{2s}e^{2\sigma \langle \xi \rangle}|\hat{f}(\xi)|^2 d\xi \\
&=  \int_{\R} \left(  \langle \xi \rangle^{2s_1\theta}e^{2\sigma\theta} |\hat{f}(\xi)|^{2\theta}\right)  \left(\langle \xi \rangle^{2(1-\theta) s_2}e^{2\sigma(1-\theta)}|\hat{f}(\xi)|^{2(1-\theta)}\right)d\xi \\
& \leq  \|J^{s_1, \sigma}u\|^{2\theta}\|J^{s_2, \sigma}u\|^{2(1-\theta)}.
\end{split}
\end{equation*}
\end{proof}

Now, we are in position to supply the proof of the main result regarding the evolution of radius of analyticity.
\begin{proof}[Proof of Theorem \ref{global}]

Let $\sigma:= \sigma(t)>0$ and $\eta_0\in G^{\sigma, 2}(\R)$. Consider $v=J^s e^{\sigma J} \eta=: J^{s, \sigma}\eta$, so that 
\begin{equation}\label{vt}
v_t =J^s e^{\sigma J} \eta_t +\sigma' J^{s+1}e^{\sigma J}  \eta. 
\end{equation}

Applying the operator $J^{s, \sigma}$ in \eqref{eq1.9}, we obtain 
\begin{equation}\label{3x11-m}
\begin{split}
v_t= \sigma' J v-i \phi(\partial_x) v -i \tau(\partial_x) J^{s, \sigma}(\eta^2)+\frac{i}8 \psi(\partial_x)J^{s, \sigma}(\eta^3)+\frac{7i}{18} \psi(\partial_x)J^{s, \sigma}(\eta_x^2).
\end{split}
\end{equation}

Multiply both sides of \eqref{3x11-m} by $v$ and then integrate in the space variable, to obtain 
\begin{equation}\label{3x11-m1}
\begin{split}
\frac12 \partial_t \int v^2=&\int  v \sigma' J v-i  \int  v \Phi(\partial_x) v -i \int v J^{s, \sigma}\tau(\partial_x) (\eta^2)+\frac{i}8 \int vJ^{s, \sigma}\psi(\partial_x)(\eta^3)\\
&+\frac{7i}{18} \int vJ^{s, \sigma}\psi(\partial_x)(\eta_x^2).
\end{split}
\end{equation}
Observe that $i \Phi(\partial_x) v =\partial_x \mathcal{K} v
= \mathcal{K}\partial_x v$, where
$$
\widehat{\mathcal{K} v}(\xi)=%K(\xi)\widehat{v}(\xi)=
\frac{1-\gamma_2\xi^2+\delta_2\xi^4}{\varphi(\xi)}\widehat{v}(\xi).
$$

Note that
\begin{equation}\label{nov3x11-m1}
i  \int  v \Phi(\partial_x) v=\int v \partial_x \mathcal{K} v=-\int (\partial_x v) (\mathcal{K} v).
\end{equation}

Using the commutativity of the operator $ \mathcal{K}$  with $\partial_x$ and the fact that  $ \mathcal{K}$  is  symmetric, one has
\begin{equation}\label{nov13x11-m1}
i  \int  v \Phi(\partial_x) v=\int v \partial_x \mathcal{K} v=\int v  \mathcal{K}  \partial_x v=\int (\mathcal{K} v )(\partial_x v).
\end{equation}

Now, combining \eqref{nov3x11-m1} and \eqref{nov13x11-m1} we conclude that 
$$
i  \int  v \Phi(\partial_x) v=0.
$$

Using the estimates from Lemmas  \ref{BT1}, \ref{Lema1}, \ref{P1}, \ref{lema2.5} and the Lemma \ref{lem31}, we get
\begin{equation}\label{3x11-m2}
\begin{split}
\frac12 \partial_t \| &J^s e^{\sigma J} \eta\|^2- \sigma'\| J^{s+1/2} e^{\sigma J} \eta\|^2 \lesssim  \| J^{s} e^{\sigma J} \eta\|^3 +\| J^{s} e^{\sigma J} \eta\|^4\\
 &\lesssim   \| J^{s} \eta\|^3 + \| J^{s} \eta\|^4+ \sigma\| J^{s+1/3} e^{\sigma J} \eta\|^3 + \sigma\| J^{s+1/4} e^{\sigma J} \eta\|^4.
\end{split}
\end{equation}

Considering the interpolation estimate in Lemma \ref{lem32}, it follows that
\begin{equation}\label{3x11-m3}
\begin{split}
\| J^{s+1/4} e^{\sigma J} \eta\| \leq \| J^{s+1/2} e^{\sigma J} \eta\|^{1/2} \| J^{s} e^{\sigma J} \eta\|^{1/2},
\end{split}
\end{equation}
and
\begin{equation}\label{3x11-m4}
\begin{split}
\| J^{s+1/3} e^{\sigma J} \eta\| &\leq \| J^{s+1/2} e^{\sigma J} \eta\|^{2/3} \| J^{s} e^{\sigma J} \eta\|^{1/3}.
\end{split}
\end{equation}

An use of  the estimates \eqref{3x11-m3} and \eqref{3x11-m4} in  \eqref{3x11-m2}, yields
\begin{equation}\label{3x11-m5}
\begin{split}
\frac12 \partial_t \| J^s e^{\sigma J} \eta\|^2&- \sigma'\| J^{s+1/2} e^{\sigma J} \eta\|^2 \lesssim  \| J^{s} \eta\|^3 + \| J^{s} \eta\|^4 \\
 &+ \sigma\| J^{s+1/2} e^{\sigma J} \eta\|^2\| J^{s} e^{\sigma J} \eta\|+\sigma\| J^{s+1/2} e^{\sigma J} \eta\|^2\| J^{s} e^{\sigma J} \eta\|^2.
\end{split}
\end{equation}
Thus
\begin{equation}\label{3x11-m6}
\begin{split}
\frac12 \partial_t \| &J^s e^{\sigma J} \eta\|^2+C\left(- \sigma'- \sigma\| J^{s} e^{\sigma J} \eta\|- \sigma\| J^{s} e^{\sigma J} \eta\|^2  \right)\| J^{s+1/2} e^{\sigma J} \eta\|^2 \lesssim \| J^{s} \eta\|^3 + \| J^{s} \eta\|^4.
\end{split}
\end{equation}

Now, if
\begin{equation}\label{3x11-m7}
\begin{split}
- \sigma'- \sigma\| J^{s} e^{\sigma J} \eta\|- \sigma\| J^{s} e^{\sigma J} \eta\|^2  = 0
\end{split}
\end{equation}
and $s=2$, then from \eqref{3x11-m6} and \eqref{consH2}, one gets
\begin{equation}\label{3x11-m8}
\begin{split}
\frac12 \partial_t \| J^2 e^{\sigma J} \eta\|^2& \lesssim \| J^{2} \eta\|^3 + \| J^{2} \eta\|^4  \sim \| J^{2} \eta_0\|^3 + \| J^{2} \eta_0\|^4.
\end{split}
\end{equation} 

From \eqref{3x11-m8} one can infer that
\begin{equation}\label{3x11-m9}
\begin{split}
\| J^2 e^{\sigma J} \eta\| \leq\| J^2 e^{\sigma_0 J} \eta_0\|+Ct^{1/2}\left(\| J^{2} \eta_0\|^{3/2} + \| J^{2} \eta_0\|^2 \right)=:\mathcal{X}_0+t^{1/2}\mathcal{Y}_0.
\end{split}
\end{equation}

The estimate \eqref{3x11-m9} and a blow-up alternative imply that if $T_{s}$ is the maximal time of existence of solution $\eta$ to the IVP \eqref{5kdvbbm}, then $T_{s}=\infty$. We prove this by using a contradiction argument. If possible suppose that $0<T_{s}<\infty$.  Let $L$ and $F$ be the linear and nonlinear parts of the IVP \eqref{5kdvbbm}. Then for any $0<T< T_s$, we have
\begin{equation}\label{1x5kdvbbm}
\begin{cases}
L \eta+F\eta=0,  \quad 0\leq t \leq T < T_s, \\
\eta(x,0)=\eta_0(x).
\end{cases} 
\end{equation}
We consider also the IVP
\begin{equation}\label{2x5kdvbbm}
\begin{cases}
L u+Fu=0,  \quad 0\leq t \leq  T_0, \\
u(x,0)=\eta(x, T).
\end{cases} 
\end{equation}
where $T_0$ is the local existence time given by \eqref{r2.45}, i.e.
$$
T_0= \frac1{8C_s\|\eta(\cdot, T)\|_{G^{\sigma,s}}(1+\|\eta(\cdot, T)\|_{G^{\sigma,s}})}.
$$ 

Using \eqref{3x11-m9}  and \eqref{consH2}, we get
\begin{equation*}%\label{3x11-m9}
\begin{split}
\| \eta(\cdot, T)\|_{G^{\sigma,2}}\leq
\| \eta_0\|_{G^{\sigma,2}}+ CT_s^{1/2}\left(\| \eta_0\|_{H^2}^{3/2} + \|  \eta_0\|_{H^2}^2 \right).
\end{split}
\end{equation*}
Thus
$$
T_0\geq  \frac1{8C_s\left( 1+\| \eta_0\|_{G^{\sigma,2}}+ CT_s^{1/2}\left(\| \eta_0\|_{H^2}^{3/2} + \|  \eta_0\|_{H^2}^2 \right) \right)^2}=:\mathcal{T}_0.
$$ 

Now, we choose $0<T<T_s$ such that
$$
T+ \mathcal{T}_0 > T_s.
$$
If $v$ is such that $u(x,t)=v(x, T+t)$, $0\leq t \leq T_0$, we obtain 
\begin{equation*}%\label{5kdvbbm}
w(x,t)=\begin{cases}
\eta(x,t),  \quad 0\leq t \leq T, \\
v(x,t),  \quad T\leq t \leq T+T_0,
\end{cases} 
\end{equation*}
is also a solution of the IVP \eqref{5kdvbbm}, with initial data $\eta_0$ in the time interval $[0, T+T_0]$ with $T+T_0> T_s$, which contradicts the definition of the maximality of $T_s$. Hence the solution is global.

Now, we move to find the lower and upper bounds for $\sigma(t)$. Note that using \eqref{3x11-m9}, the estimate  \eqref{3x11-m7} is true if 
\begin{equation}\label{3x11-m10}
\begin{split}
- \sigma' &=  \sigma\| J^{2} e^{\sigma J} \eta\|+ \sigma\| J^{2} e^{\sigma J} \eta\|^2 \\
&\leq  \sigma\left( \mathcal{X}_0+t^{1/2}\mathcal{Y}_0+2\mathcal{X}_0^2+2t\mathcal{Y}_0^2 \right)  =:\sigma \, \mathcal {A}(t).
\end{split}
\end{equation}

 On the other hand, the estimate \eqref{3x11-m10}  is equivalent to
\begin{equation}\label{3x11-m11}
\begin{split}
\sigma(t)& \geq  \sigma_0 e^{-\int_0^t \mathcal {A}(t') dt'}\\
&= \sigma_0 e^{-(\mathcal{X}_0+2\mathcal{X}_0^2)t-\frac32 t^{3/2}\mathcal{Y}_0-t^2 \mathcal{Y}_0^2 },
\end{split}
\end{equation}
where $\sigma_0= \sigma(0)$. This  provides the lower bound for $\sigma(t)$.

 Now, we proceed  to find an upper bound $\sigma(t)$. Considering \eqref{3x11-m10} and \eqref{consH2}, we have
\begin{equation*}\label{3x11-m13}
\begin{split}
-\sigma' &\geq  \sigma\| J^{2} e^{\sigma J} \eta\|^2\\
& \geq  \sigma\| J^{2}  \eta\|^2\\
& \gtrsim \sigma\| J^{2}  \eta_0 \|^2.
\end{split}
\end{equation*}
Consequently
\begin{equation*}\label{3x11-m14}
\begin{split}
\sigma(t) \leq  C \sigma_0 e^{-\| J^{2}  \eta_0 \|^2 t}.
\end{split}
\end{equation*}
\end{proof}

\section*{Acknowledgments}
The first author extends thanks to PRONEX-FAPERJ for the support received. The second author acknowledges the grants from FAPESP (2020/14833-8) and CNPq (307790/2020-7), and  is also thankful to the IM-UFRJ for pleasant hospitality where a part of this work was developed.

% References


\begin{thebibliography}{00}


\small




\bibitem{BHG-17} R. F. Barostichi, A. A. Himonas, G. Petronilho; {\em Global analyticity for a generalized Camassa-Holm equation and decay of the radius of spatial analyticity} J. Differential Equations {\bf 263} 1 (2017) 732--764.


%\bibitem{BBM}  T. B. Benjamin, J. L. Bona and J. J. Mahony;   {\em Model %equations for long waves in nonlinear dispersive media}, Philos. Trans. Royal Soc. %London, Series A {\bf 272} (1972), 47--78.  


\bibitem{BCG}J. L. Bona, H. Chen, C. Guillop\'e, {\em Further Theory for a Higher-order Water Wave Model}, to appear in Jr. Pure and Applied 
Functional Analysis.

\bibitem{BCPS1} J. L. Bona, X. Carvajal, M.  Panthee, M. Scialom; {\em Higher-Order Hamiltonian Model for Unidirectional Water Waves},  J. Nonlinear Sci. {\bf 28} (2018) 543--577.


\bibitem{BCG}J. L. Bona, H. Chen, C. Guillop\'e, {\em Further Theory for a Higher-order Water Wave Model},  Jr. Pure and Applied  Functional Analysis {\bf 4} Number 4 (2019) 685--708.


%\bibitem{minbona}  J. L. Bona, M.  Chen;  {\em A Boussinesq system for two-%way propagation of nonlinear dispersive waves}, Physica D {\bf 116}  (1998) %191--224. 

\bibitem{BCS1} J. L. Bona, M. Chen and J.-C. Saut; {\em Boussinesq equations and other systems for small-amplitude long waves in nonlinear dispersive media I. Derivation and linear theory}, J. Nonlinear Sci. {\bf 12} (2002) 283--318.

\bibitem{BCS2}  J. L. Bona, M. Chen and J.-C. Saut;  {\em Boussinesq equations and other systems for small-amplitude long waves in nonlinear dispersive media II. The nonlinear theory}, Nonlinearity {\bf 17} (2004) 925--952.


%\bibitem{BK}  J. L. Bona and H. Kalisch;  {\em  Models for internal waves in deep %water},  Discrete Contin. Dyn. Syst. {\bf 6} (2000) 1--20.

%\bibitem{BPS}  J. L. Bona, W. G.  Pritchard and L. R. Scott;  {\em An Evaluation %of a Model Equation for Water Waves},
%Philos. Trans. Royal Soc. London, Series A.  {\bf 302} (1981) 457--510.



\bibitem{BT} J. L. Bona and N. Tzvetkov; {\em Sharp well-posedness results for the BBM equation},  Discrete Continuous Dynamical Systems, Ser. A  {\bf 23} (2009) 1241--1252.



\bibitem{BGK-1} J. L. Bona, Z. Gruji\'c, H. Kalisch; {\em Algebraic lower bounds for the uniform radius of spatial analyticity for the generalized KdV equation}, Ann. I.H. Poincar\'e, {\bf AN22} (2005) 783--797.

\bibitem{BGK-2} J. L. Bona, Z. Gruji\'c, H. Kalisch; {\em Global solutions for the derivative Schr\"odinger equation in a class of functions analytic in a strip}, J. Differential Equations {\bf 229} (2006) 186--203.


\bibitem{BG-1} J. L. Bona, Z. Gruji\'c; {\em Spatial analyticity properties of nonlinear waves}, Mathematical Models and Methods in Applied Sciences {\bf 13} (2003) 345--360.

\bibitem{BHK} A. de Bouard, N. Hayashi N., K. Kato;  {\em Gevrey regularizing effect for the (generalized) Korteweg-de Vries equation and nonlinear Schr\"odinger equations}, Ann. Inst. H. Poincar\'e Anal. Non Lin\'eaire, {\bf 6} (1995) 673--715.



\bibitem{CPP-1} X. Carvajal, M. Panthee, R. Pastran; {\em On the well-posedness, ill-posedness and norm-inflation for a higher order water wave model on a periodic domain},  Nonlineay Analysis TMA, {\bf 192}  (2020) 1--22.


\bibitem{CP} X. Carvajal, M. Panthee;  {\em On sharp global well-posedness and Ill-posedness for a fifth-order KdV-BBM type equation}, Jr. Math. Anal. Appl. {\bf 479} (2019) 688--702.



\bibitem{hongqiu}   H. Chen;  {\em Well-posedness for a higher-order, nonlinear, dispersive equation on a quarter plane}, Discrete Cont. Dynamical 
Systems, Ser. A, {\bf 38} (2018) 397--429.   




\bibitem{FT} C. Foias, R. Temam; {\em Gevrey class regularity for the solutions of the Navier-Stokes equations}, J. Funct. Anal. {\bf 87} (1989) 359--369.




\bibitem{GK-1} Z. Gruji\'c, H. Kalisch; {\em Local well-posedness of the generalized Korteweg-de Vries equation in spaces of analytic functions}, Differential and Integral Equations, {\bf 15} (2002) 1325--1334.

\bibitem{GK-2} Z. Gruji\'c, H. Kalisch; {\em The Derivative Nonlinear Schr\"odinger Equation in Analytic Class}, J. Nonlinear Math. Phy, {\bf 10} (2003) 62--71.

\bibitem{H} N. Hayashi ; {\em Global existence of small analytic solutions to nonlinear Schr\"odinger equations}, Duke Math. J., {\bf 60} (1990) 717--727.

\bibitem{HO} N. Hayashi, T. Ozawa; {\em Remarks on nonlinear Schr\"odinger equations in one space dimension}, Differential Integral Equations, {\bf 7} (1994) 453--461.


\bibitem{HP-20} A. A. Himonas,  G. Petronilho; {\em Evolution of the radius of spatial analyticity for the periodic BBM equation},  Proc. Amer. Math. Soc. {\bf 148} 7 (2020) 2953--2967.

\bibitem{HP-18} A. A. Himonas,  G. Petronilho; {\em Radius of analyticity for the Camassa-Holm equation on the line} Nonlinear Anal. {\bf 174} (2018) 1--16.

 \bibitem{HPS-17} A. A. Himonas,  G. Petronilho, S. Selberg; {\em On persistence of spatial analyticity for the dispersion-generalized periodic KdV equation} Nonlinear Anal. Real World Appl. {\bf 38} (2017) 35--48.
 
 
\bibitem{ILP-15} P. Isaza, F. Linares, and G. Ponce, {\em On the propagation of regularity and decay of solutions to the k-generalized Korteweg-de Vries equation}, Comm. Partial Differential Equations 40 (2015), no. 7, 1336–1364.

\bibitem{ILP-A1} P. Isaza, F. Linares, and G. Ponce, {\em On the propagation of regularities in solutions of the Benjamin-Ono equation},  J. Funct. Anal. {\bf 270}  (2016) 976--1000.

\bibitem{ILP-A2} P. Isaza, F. Linares, and G. Ponce, {\em On the propagation of regularity of solutions
of the Kadomtsev-Petviashvilli  equation}, SIAM J. Math. Anal. {\bf 48} (2016) 1006--1024.


\bibitem{KM} T. Kato T., K. Masuda; {\em Nonlinear evolution equations and analyticity I}, Ann. Inst. H. Poincar\'e Anal. Non Lin\'eaire, {\bf 3} (1986) 455--467.

\bibitem{KO} T. Kato, T. Ogawa; {\em Analyticity and smoothing effect for the Korteweg-de Vries equation with a single point singularity}, Math. Ann., {\bf 316} (2000) 577--608.

\bibitem{LPS-A1} F. Linares, G. Ponce, D. L. Smith, {\em On the regularity of solutions to a class of nonlinear dispersive equations}, Math. Ann. {\bf 369} (2017), no. 1-2, 797--837.


\bibitem{SS-19} S. Selberg; {\em  On the radius of spatial analyticity for solutions of the Dirac-Klein-Gordon equations in two space dimensions}, Ann. Inst. H. Poincaré Anal. Non Linéaire {\bf 36} 5 (2019) 1311--1330.

\bibitem{SS-18} S. Selberg; {\em Spatial analyticity of solutions to nonlinear dispersive PDE}, Non-linear partial differential equations, mathematical physics, and stochastic analysis, 437--454, EMS Ser. Congr. Rep., Eur. Math. Soc., Zürich, 2018.

\bibitem{ST-17} S. Selberg, A. Tesfahun; {\em On the radius of spatial analyticity for the quartic generalized KdV equation},  Ann. Henri Poincaré 18 (2017), no. 11, 3553--3564.

\bibitem{SdS-15} S. Selberg,  D. O. da Silva; {\em A remark on unconditional uniqueness in the Chern-Simons-Higgs model}, Differential Integral Equations {\bf 28} 3-4 (2015) 333--346.

\bibitem{Z1} B. Y. Zhang, {\em Taylor series expansion for solutions of the KdV equation with respect to
their initial values}, J. Funct. Anal. {\bf 129} (1995) 293--324.

\bibitem{Z2} B. Y. Zhang, {\em Analyticity of solutions of the generalized KdV equation with respect to
their initial values}, SIAM J. Math. Anal. {\bf 26} (1995) 1488--1513.

    
\end{thebibliography}
\end{document}